\documentclass[11pt]{amsart}
\usepackage[babel]{csquotes}
\usepackage{enumitem}
\usepackage{amsmath,amsthm,amssymb,mathrsfs,amsfonts,verbatim,enumitem,color,leftidx}
\usepackage{mathabx}
\usepackage{etoolbox} 
\usepackage{bbm}
\usepackage[all,tips]{xy}
\usepackage{graphicx,ifpdf}
\usepackage{stmaryrd}
\ifpdf
\fi
\usepackage[colorlinks]{hyperref}
\hypersetup{
linkcolor=blue,          
citecolor=green,        
}

\usepackage{tikz}

\newtheorem{thm}{Theorem}[section]

\newtheorem{cor}[thm]{Corollary}
\newtheorem{prop}[thm]{Proposition}

\theoremstyle{definition}
\newtheorem{defi}[thm]{Definition}

\theoremstyle{remark}
\newtheorem{rem}[thm]{Remark}

\numberwithin{equation}{section}
\definecolor{esperance}{rgb}{0.0,0.5,0.0}

\newcommand{\bu}{\mathbf{u}}

\newcommand{\bw}{\mathbf{w}}

\newcommand{\R}{\mathbb{R}}
\newcommand{\N}{\mathbb{N}}
\newcommand{\Z}{\mathbb{Z}}

\newcommand{\s}{\sigma}
\newcommand{\de}{\delta}

\DeclareMathOperator{\diam}{diam}

\newcommand{\del}{\delta}

\newcommand{\eps}{\epsilon}

\newcommand{\sig}{\sigma}

\newcommand{\cA}{\mathcal{A}}
\newcommand{\cB}{\mathcal{B}}

\newcommand{\cL}{\mathcal{L}}

\newcommand{\cP}{\mathcal{P}}

\newcommand{\bR}{\mathbb{R}}
\newcommand{\bZ}{\mathbb{Z}}

\newcommand{\bN}{\mathbb{N}}

\newcommand\set[1]{\left\{#1\right\}}

\newcommand\idist[1]{\langle#1\rangle}

\newcommand\on[1]{\operatorname{#1}}

\newcommand\mb[1]{\mathbf{#1}}

\newcommand\crly[1]{\mathscr{#1}}

\newcommand{\wstar}{\overset{\on{w}^*}{\lra}}
\newcommand{\Supp}{\on{Supp}}

\newcommand{\defn}{\overset{\on{def}}{=}}

\newcommand{\lra}{\longrightarrow}

\newcommand{\onto}{\xymatrix{\ar@{>>}[r]&}}

\newcommand{\eqlabel}[2]
{
\begin{equation}
{#2}\label{#1}
\end{equation}
}

\begin{document}

\title{Dimension bound in inhomogeneous Diophantine approximation}

\date{April 22, 2019}

\author{Wooyeon Kim}
\address{Wooyeon Kim. Department of Mathematical Sciences, Seoul National University, 
{\it wooyeon817@snu.ac.kr}}
\author{Seonhee Lim}
\address{S.~Lim. Department of Mathematical Sciences and Resesarch Institute of Mathematics, Seoul National University,
{\it slim@snu.ac.kr}}

\thanks{}



\def\thefootnote{}
\def\thefootnote{\arabic{footnote}}
\setcounter{footnote}{0}

\begin{abstract}
We prove that for all $b$, the Hausdorff dimension of the set of $m \times n$ matrices $\epsilon$-badly approximable for the target $b$ is not full. The doubly metric case follows. 

It was known that for almost every matrix $A$, the Hausdorff dimension of the set $Bad_A(\epsilon)$ of $\epsilon$-badly approximable target $b$ is not full, and that for dimension 1, i.e. for real numbers $\alpha$, $\dim_H Bad_\alpha(\epsilon)=1$ if and only if $\alpha$ is singular on average. We show that if $\dim_H Bad_A(\epsilon)=m$, then $A$ is singular on average.

\end{abstract}

\maketitle
\section{Introduction}
In Diophantine approximation, one wants to approximate an irrational number $\alpha$ by rationals $p/q$ for $p,q \in \mathbb{Z}$. By pigeonhole principle or Dirichlet theorem, for every $N \in \mathbb{N}$, there exists $p,q \in \mathbb{Z}$ with $0<q<N$, such that
$$|q\alpha-p|<1/N < 1/q.$$
As above, one can see classical (homogeneous) Diophantine approximation as studying distribution of $q\alpha$ module $\mathbb{Z}$ near zero. Diophantine approximation for irrational numbers has been generalized to studying vectors, linear forms, and more generally matrices, and are classical subjects in metric number theory. 

In this article, we consider the inhomogenous Diophantine approximation: we still study the distribution of $q\alpha$ modulo $\mathbb{Z}$ but now near a ``target" $b \in \mathbb{R}$. Again, for every $N \in \mathbb{N}$, there exists $p,q \in \mathbb{Z}$ with $0<q<N$, such that
$$|q\alpha-b -p |< 1/q.$$
Let $M_{m,n}(\bR)$ be the set of $m\times n$ real matrices, and let $\widetilde{M}_{m,n}(\bR) := M_{m,n}(\bR)\times \bR^m$. Similarly to numbers, for $A\in M_{m,n}(\bR)$, we study $Aq \in \R^m$  modulo $\mathbb{Z}^m$ near the target $b \in \mathbb{R}^m$ for vectors $q \in \mathbb{Z}^n$.
For $v \in \bR^m$, denote by $\idist{v}: =\displaystyle\inf_{p\in \bZ^m} ||v-p||$ the distance from $v$ to the nearest integral vector. In this general situation as well, not only Dirichlet theorem holds but we also have
$$\liminf_{q\in\bZ^n, ||q||\to \infty} ||q||^{n/m}\idist{Aq-b}=0,$$ 
for almost every $(A,b) \in M_{m,n}(\bR)\times \bR^m.$ 

The exceptional set of the above equality is our object of interest.
We call $A$ $\eps$-\textit{bad} for $b\in\mathbb{R}^m$ if
\eqlabel{eq1523}{
\liminf_{q\in\bZ^n, ||q||\to \infty} ||q||^{n/m}\idist{Aq-b}\ge \eps
.}
 Denote 
\begin{align*}
\mb{Bad}(\eps)&\defn\set{(A,b)\in \widetilde{M}_{m,n}(\bR):A\textrm{ is $\eps$-bad for $b$}},\\ 
\mb{Bad}_A(\eps)&\defn\set{b\in\bR^m:A\textrm{ is $\eps$-bad for $b$}}, \;\;\mb{Bad}_A\defn\bigcup_{\eps>0}\mb{Bad}_A(\eps),\\
\mb{Bad}^b(\eps)&\defn\set{A\in M_{m,n}(\bR):A\textrm{ is $\eps$-bad for $b$}}, \;\; \mb{Bad}^b\defn\bigcup_{\eps>0}\mb{Bad}^b(\eps),\\
\mb{Bad}^{'}(\eps)&\defn\bigcup_{b\in\bR^m}\mb{Bad}^b(\eps).
\end{align*}

The set $\mb{Bad}^0$ can be seen as the set of badly approximable systems of $m$ linear forms in $n$ variables. This set is of Lebesgue measure zero \cite{Gro38}, but has full Hausdorff dimension $mn$ \cite{Sch69}.

For any $b$, $\mb{Bad}^b$ also has zero Lebesgue measure \cite{Sch} and full Hausdorff dimension for every $b$ \cite{ET}. Indeed, it is shown that $\mb{Bad}^b$ is a winning set \cite{ET}  and even  a hyperplane winning set \cite{HKS}, a property which implies full Hausdorff dimension. On the other hand, the set $\mb{Bad}_A$  also has full Hausdorff dimension for every $A$ \cite{BHKV10}, but can have positive Lebesgue measure.

The set $\mb{Bad}^b$ and $\mb{Bad}_A$ are unions of subsets $\mb{Bad}^b(\eps)$ and $\mb{Bad}_A(\eps)$ over $\eps>0$, respectively, thus a more refined question is about the Hausdorff dimension of $\mb{Bad}^b(\eps)$, $\mb{Bad}_A(\eps)$. For the homogeneous case $b=0$, the Hausdorff dimension $\mb{Bad}^0(\eps)$ is less than the full dimension $mn$ \cite{BK13} (see also \cite{Sim} for more precise estimation). 

Thus, a natural question is whether $\mb{Bad}^b(\eps)$ can have full Hausdorff dimension for general $b$. 
Our main result answers the above question.

\begin{thm}\label{thm0}
For any $\eps>0$, $\dim_H \mb{Bad}^{'}(\eps)<mn.$
\end{thm}

This result directly implies an upper bound for each $\mb{Bad}^b(\eps)$'s.

\begin{cor}\label{thm1}
For any $\eps>0$, there exists $\del>0$ such that for all $b\in\bR^m$, $$\dim_H \mb{Bad}^b(\eps)<mn-\del.$$
\end{cor}

Theorem \ref{thm0} also directly implies a similar result for doubly metric case.

\begin{cor}\label{cor1}
For any $\eps>0$, there exists $\del>0$ such that $$\dim_H \mb{Bad}(\eps)<mn+m-\del.$$
\end{cor}

It was known that the set $\mb{Bad}_A(\eps)$ has Hausdorff dimension less than the full dimension $m$ for almost every $A$ \cite{LSS}. The argument for Corollary \ref{thm1} can be applied similarly to improve the result for $\mb{Bad}_A(\eps)$ in \cite{LSS} in terms of the exceptional set.

\begin{thm}\label{thm2}
For any $\eps>0$, there exists $\del>0$ such that the Hausdorff dimension of the set of $A$ satisfying $\dim_H \mb{Bad}_A(\eps)\ge m-\del$ is strictly less than $mn$.
\end{thm}

To state our last result, let us introduce more notations. For $d=m+n$, let $G(\R)=ASL_d(\bR)$ be the set of area-preserving affine transformations. Let $G(\bZ)=ASL_d(\bZ)=SL_d(\bZ)\ltimes\bZ^d=Stab_{G(\bR)}(\bZ^d)$ be the stabilizer of $\bZ^d$.
For the 1-parameter diagonal subgroup $$\set{a_t=\Delta (e^{nt}\mathbf{1}_m,e^{-mt}\mathbf{1}_n)}_{t \in \mathbb R}$$ in $SL_d(\bR)$, we take a lift of this group to $G(\bR)\subset SL_{d+1}(\bR)$ given by 
$a_t\longmapsto\left(\begin{matrix}
a_t & 0\\
0 & 1\\
\end{matrix}\right)$
and we denote it again by $a_t$ by abuse of notation. Let $a\defn a_1$ be the time-one map of the diagonal flow $a_t$. We denote by $U$ the subgroup $\set{u(A):=
\left(\begin{matrix}
I_m & A & 0\\
0 & I_n & 0\\
0 & 0 & 1\\
\end{matrix}\right)
:A\in M_{m,n}(\bR)}$ and $W$ the subgroup $\set{w(b):=
\left(\begin{matrix}
I_m & 0 & b\\
0 & I_n & 0\\
0 & 0 & 1\\
\end{matrix}\right)
:b\in \bR^m}$ of $G(\bR)$, both of which are unstable horospherical subgroups for $a$.

Let $Y\defn G(\R)/G(\bZ)$, which can be seen as the space of unimodular grids, i.e. unimodular lattices translated by a vector in $\bR^d$. Let $\bf{d}(\cdot,\cdot)$ be a right invariant metric on $Y$. We may assume that the locally defined maps (on balls of radius, say $r_0$,) log and exponential maps between $U$ and its Lie algebra $\bu$ are Lipschitz. For the norm of Lie algebra $\bu$, we take the Euclidean norm $||\cdot||_{\bR^{mn}}$. Similarly, we may assume that log and exponential maps between balls of $W$ and its Lie algebra $\bw$ are bi-Lipshitz if we take the Euclidean norm $||\cdot||_{\bR^m}$ for the norm of Lie algebra $\bw$.  Observe that $Ad_{a_t}u=e^{(m+n)t}u$ holds for any $u\in \bu$ and $Ad_{a_t}w=e^{nt}w$ for any $w\in\bw$.
Let $X=SL(d,\bR)/SL(d,\bZ)$ be the space of unimodular lattices. There is a natural projection $\pi:Y\to X$ sending a translated lattice $x +b$ to the corresponding lattice $x$. It is defined by $\pi(
\left(\begin{matrix}
B & v\\
0 & 1\\
\end{matrix}\right)G(\bZ))
=B\cdot SL(d,\bZ)$ for $B\in SL(d,\bR)$ and $v\in \bR^d.$

For $A\in M_{m,n}$, we associate a point $x_A=
\left(\begin{matrix}
I_m & A\\
0 & I_n\\
\end{matrix}\right)SL(d,\bZ)$ in $X$. For $(A,b)\in\widetilde{M}_{m,n}(\bR)$, we associate a point $x_{A,b}=
\left(\begin{matrix}
I_m & A & -b\\
0 & I_n & 0\\
0 & 0 & 1\\
\end{matrix}\right)y_0$ in $Y$, where $y_0$ is the identity coset $SL_d(\bZ)\ltimes\bZ^d$.

We say that a point $x\in X$ has \emph{$\del$-escape of mass on average} (with respect to the diagonal flow $a_t$) if for any compact set $Q$ in $X$,
$$\displaystyle\liminf_{N\to\infty}\frac{1}{N}|\set{l\in\set{1,\dots,N}: a_l x\notin Q}|\ge\del.$$

In \cite{LSS}, it was shown that $\dim_H \mb{Bad}_A(\eps)<m$ for all $\eps>0$ if $x_A$ is $\textit{heavy}$ which is a condition equivalent to no escape of mass on average. Note that $x_A$ is heavy for almost every $A\in M_{m,n}(\R)$. A $m\times n$ matrix $A$ is called $\textit{singular on average}$ if for any $\eps>0$
$$\lim_{N\to\infty}\frac{1}{N}|\set{l\in\set{1,\cdots,N}:\exists q\in\Z^n \ s.t. \ 
\idist{Aq}<\eps 2^{-\frac{nl}{m}} \ \textrm{and} \ 0<||q||<2^l}|=1.$$
This property is equivalent to the fact that the corresponding point $x_A$ has $1$-escape of mass on average (with respect to the diagonal flow $a_t$) by Dani's correspondence.

For $m=n=1$, $A$ is singular on average if and only $\mb{Bad}_A(\eps)$ has full Hausdorff dimension for some $\eps>0$ \cite{BKLR}. 
For  $(m,n)\neq (1,1)$, nothing was known about $\dim_H \mb{Bad}_A(\eps)$ for $A$ which has $\eta$-escape of mass on average for some $0<\eta<1$. The next theorem deals with this case.

\begin{thm}\label{thm3}
Let $A\in M_{m,n}(\R)$. If $A$ has $\delta$-espape of mass on average for some $\delta <0$ (but not divergent on average), then
 $\dim_H \mb{Bad}_A(\eps)<m$ for any $\eps>0$. \end{thm}
We remark that the set of matrices which are singular on average has Hausdorff dimension at most $mn-\frac{m+n}{mn}$ \cite{KKLM}. We also remark that the case where the best approximation vectors $Q_k$ of $A$ satisfies $\lim_{k \to \infty} |Q_k|^{1/k} = \infty$, then $\dim_H \mb{Bad}_A(\eps)=m$ for some $\eps>0$ \cite{BKLR}. Thus the only remaining case where we do not know whether $\dim_H   \mb{Bad}_A(\eps)$ can be full is the case where $A$ is singular on average but $ |Q_k|^{1/k}$ is bounded.

\section{Preliminaries}
\subsection{Correspondence with dynamics}
For $y\in Y$, $\Lambda_y$ denotes the corresponding unimodular grid in $\bR^d$.
Let
$$\cL_\eps\defn\set{y : \Lambda_y \cap B_{\eps^{\frac{m}{m+n}}}^{\bR^d}(0)=\phi},$$
which is a (non-compact) closed subset of $Y.$
\begin{prop}\label{prop1}
For any $(A,b)\in \mb{Bad}(\eps)$, one of the following statements holds:\\
(1) there exists some $q \in \bZ^n$ such that 
\eqlabel{eqn0}{\idist{Aq-b}=0.}\\
(2) the $a_t$-orbit of the point $x_{A,b}$ is eventually in $\cL_\eps$, i.e., there exists $T\ge 0$ such that $a_t x_{A,b}\in \cL_\eps$ for all $t\ge T$.
\end{prop}

\begin{proof}
Assume that both of the statements do not hold. Then there exist arbitrarily large $t$'s satisfying $a_t x_{A,b}\notin \cL_\eps$. As
$$a_t x_{A,b}=
\left(\begin{matrix}
e^{nt}I_m & 0 & 0 \\
0 & e^{-mt}I_n & 0\\
0 & 0 & 1\\
\end{matrix}\right)
x_{A,b}=
\left(\begin{matrix}
e^{nt}I_m & e^{nt}A & -e^{nt}b\\
0 & e^{-mt}I_n & 0\\
0 & 0 & 1\\
\end{matrix}\right)x_0,
$$
the vectors in the grid $\Lambda_{a_t x_{A,b}}$ can be represented as $\left(\begin{matrix}
 e^{nt}(Aq+p-b)\\
 e^{-mt}q\\
\end{matrix}\right)$ for integer vectors $(p,q)\in\bZ^m\times\bZ^n$.
Therefore $a_t x_{A,b}\notin \cL_\eps$ implies that for some $q$, $\idist{Aq-b}e^{nt}<\eps^{\frac{m}{m+n}}$ and $e^{-mt}||q||<\eps^{\frac{m}{m+n}}$, thus $||q||^{\frac{n}{m}}\idist{Aq-b}<\eps$. Since $\idist{Aq-b}\neq0,\forall q$, we use the condition $\idist{Aq-b}e^{nt}<\eps^{\frac{m}{m+n}}$ for arbitrarily large $t$ to conclude that $||q||^{\frac{n}{m}}\idist{Aq-b}<\eps$ holds for infinitely many $q$'s. This is a contradiction to the assumption that $(A,b)\in \mb{Bad}(\eps)$.
\end{proof}

We claim that for a fixed $b\in \bR^m$, the subset $\mb{Bad}_{(1)}^b(\eps)$ of $\mb{Bad}^b(\eps)$ satisfying \eqref{eqn0} is a subset of $\mb{Bad}^0(\eps).$ Indeed, if $A\in \mb{Bad}^b(\eps)$ for some $b$ and satisfies \eqref{eqn0}, $\idist{Aq_0-b}=0$ for some $q_0\in\bZ^m$ and $\displaystyle\liminf_{||q||\to \infty} ||q||^{n/m}\idist{Aq-b}\ge \eps$, thus $\displaystyle\liminf_{||q||\to \infty} ||q||^{n/m}\idist{A(q-q_0)}\ge \eps$. Therefore, for
$\mb{Bad}^{'}_{(1)}(\eps)\defn\displaystyle\bigcup_{b\in\bR^m}\mb{Bad}^{b}_{(1)}(\eps)$, we have $$\dim_H\mb{Bad}_{(1)}^{'}(\eps)\leq\dim_H\mb{Bad}^0(\eps)=mn-c_{m,n}\eps^m+o(\eps^m)<mn$$
for some constant $c_{m,n}>0$ \cite{Sim}. 

For a fixed $A\in M_{m,n}(\bR)$, the subset of $\mb{Bad}_A(\eps)$ satisfying \eqref{eqn0} is of the form $Aq+p$ for some $q,p\in\bZ^m$ thus has Hausdorff dimension zero.

In the rest of the article, we will focus on $x_{A,b}$ that are eventually in $\cL_\eps$.

\subsection{Entropy and relative entropy}
In this subsection, we recall the definition and basic properties of the entropy and the relative entropy for $\sigma$-algebras we used in the later sections. We refer the reader to \cite{ELW} for details.

\begin{defi}
Let $(X,\cB,\mu)$ be a probability space and $\cP$ be a finite or countable partition of $(X,\cB,\mu)$. 
\begin{enumerate}
\item
The \emph{(static) entropy} of the partition $\cP=\set{A_1,A_2,\dots}$ is
$$H_\mu(\cP)=H(\mu(A_1),\dots)=-\displaystyle\sum_{i\ge 1}\mu(A_i)\log\mu(A_i)\in[0,\infty]$$
where $0\log0=0$.
\item
Let $T$ be a measuere-preserving transformation of $(X,\cB,\mu)$ and assume $\cP$ is a partition of $X$ with finite entropy. Then the \emph{(dynamical) entropy of $T$ with respect to $\cP$} is
$$h_\mu(T,\cP)=\displaystyle\lim_{n\to\infty}\frac{1}{n}H_\mu\Bigl(\displaystyle\bigvee_{i=0}^{n-1}T^{-i}\cP\Bigr)
=\displaystyle\inf_{n\ge 1}\frac{1}{n}H_\mu\Bigl(\displaystyle\bigvee_{i=0}^{n-1}T^{-i}\cP\Bigr).$$
The \emph{(dynamical) entropy of $T$} is
$$h_\mu(T)=\displaystyle\sup_{\cP:H_\mu(\cP)<\infty}h_\mu(T,\cP).$$
\end{enumerate}
\end{defi}

Basic and important properties of the entropy include subadditivity, concavity and linearity.
Let $\cP$ and $\cP'$ be countable partitions of a measure-preserving system $(X,\cB,\mu,T)$. Then,

 $$H_\mu(\cP\vee\cP')\leq H_\mu(\cP)+H_\mu(\cP'),$$
and $$h_\mu(T,\cP\vee\cP')\leq h_\mu(T,\cP)+h_\mu(T,\cP').$$

Let $\cP$ be a partition of $(X,\cB, \mu_i, T)$ which are measure-preserving systems for $1\leq i\leq k$. Let $\mu=\sum_{i} a_i\mu_i$ be a convex combination of $\set{\mu_i}$ with $0\leq a_i\leq 1$ and $\sum_{i} a_i=1$. Then,
\begin{enumerate}
    \item (concavity for static entropy) $H_\mu(\cP)\leq \displaystyle\sum_{i=1}^k a_iH_{\mu_i}(\cP),$
    \item (linearity for dynamical entropy) $h_\mu(T)= \displaystyle\sum_{i=1}^k a_ih_{\mu_i}(T)$.
\end{enumerate}

In section~\ref{sec:entropybound}, we also need a relative (conditional) entropy with respect to a subalgebra to obtain subgroup invariance.

\begin{defi}
Let $(X,\cB,\mu,T)$ be a measure-preserving system and let $\cA\subseteq\cB$ be a countably-generated sub-$\sigma$-algebra. Note that there exists an $\cA$-measurable conull set $X'\subset X$ and a system $\set{\mu_x^\cA|x\in X'}$ of measures on $X$, referred to as \emph{conditional measures}. The \emph{information function} of a countable partition $\cP$ given $\cA$ with respect to $\mu$ is defined by
$$I_\mu(\cP|\cA)(x)=-\log\mu_x^\cA([x]_\cA),$$
where $[x]_\cA$ is the smallest element of $\cA$ containing $x$. Moreover, the \emph{relative (static) entropy of $\cP$ given $\cA$} is defined by
$$H_\mu(\cP|\cA)=\int I_\mu(\cP|\cA)(x)d\mu(x),$$
which is the average of the information. If $\cA$ is a strictly invariant sub-$\sigma$-algebra, i.e. $T^{-1}(\cA)=\cA$, then the \emph{relative (dynamical) entropy of $T$ given $\cA$} is
$$h_\mu(T|\cA)=\displaystyle\sum_{\cP:H_\mu(\cP)<\infty}h_\mu(T,\cP|\cA)$$
where
$$h_\mu(T,\cP|\cA)=\displaystyle\lim_{n\to\infty}\frac{1}{n}H_\mu(\displaystyle\bigvee_{i=0}^{n-1}T^{-i}\cP|\cA)
=\displaystyle\inf_{n\ge 1}\frac{1}{n}H_\mu(\displaystyle\bigvee_{i=0}^{n-1}T^{-i}\cP|\cA)$$
for any countable partiton $\cP$.
\end{defi}
Subadditivity, concavity and linearity of relative entropy also hold.
\begin{prop}
(Subadditivity) Let $\cP$ and $\cP'$ be countable partitions of a measure-preserving system $(X,\cB,\mu,T)$ and let $\cA\subseteq\cB$ be a countably-generated sub-$\sigma$-algebra. Then,
\begin{enumerate}
    \item $H_\mu(\cP\vee\cP'|\cA)\leq H_\mu(\cP|\cA)+H_\mu(\cP'|\cA),$
    \item $h_\mu(T,\cP\vee\cP'|\cA)\leq h_\mu(T,\cP|\cA)+h_\mu(T,\cP'|\cA).$
\end{enumerate}
\end{prop}
\begin{proof}
We refer the reader to \cite{ELW}, Proposition 2.13 for the proof.
\end{proof}

\begin{prop}
Let $\cP$ be a partition of $(X,\cB, \mu_i, T)$ which are measure-preserving systems for $1\leq i\leq k$, $\mu=\displaystyle\sum_{i=1}^k a_i\mu_i$ be a convex combination of $\set{\mu_i}$ with $0\leq a_i\leq 1, \displaystyle\sum_{i=1}^k a_i=1$, and $\cA$ be a strictly invariant sub-$\sigma$-algebra. Then,
\begin{enumerate}
    \item (concavity for static entropy) $H_\mu(\cP|\cA)\leq \displaystyle\sum_{i=1}^k a_iH_{\mu_i}(\cP|\cA),$
    \item (linearity for dynamical entropy) $h_\mu(T|\cA)= \displaystyle\sum_{i=1}^k a_ih_{\mu_i}(T|\cA)$.
\end{enumerate}
\end{prop}
\begin{proof}
The concavity of the function $f(x)=-x\log x$ on $[0,1]$ directly implies (1). See Theorem 2.33  \cite{ELW} for the proof of (2).
\end{proof}

\section{Constructing measure with entropy lower bound}\label{sec:entropybound}
In this subsection, we construct an $a_t$-invariant measure on $Y$ with a lower bound on the relative entropy. 

We will use entropy contribution for the subgroup $U$, with respect to the Borel $\sig$-algebra $\cB_Y^U$ generated by the collection of $U$-invariant Borel sets in $Y$. Before constructing the desired measure, we recall the following theorem about escape of mass.
\begin{thm}[KKLM, Remark 2.1]\label{KKLM}
For any $x\in X$, the set
$$\set{u\in U| ux\ \textrm{has} \ \del\textrm{-escape of mass on average}}$$
has Hausdorff dimension at most $mn-\frac{\del(m+n)}{mn}$.
\end{thm}

\begin{rem}
For the definition of escape of mass, we follow the definition of singularity on average in \cite{DFSU}, Section 1.4. \cite{KKLM} uses $\lim$ instead of $\liminf$ for the definition of escape of mass.
\end{rem}

 For any compact set $Q\subset X$ and positive integer $k>0$,
and any $0<\eta<1$, let
\begin{align*}
 E_\eta&\defn\set{A\in M_{m,n}(\bR): x_A \ \textrm{has} \ \eta\textrm{-escape of mass on average}}\\
 F_{\eta,Q}&\defn\set{A\in M_{m,n}(\bR):\frac{1}{k}\displaystyle\sum_{i=0}^{k-1}\del_{a^i x_A}
 (X\setminus Q)<\eta \  \textrm{for infinitely many} \  k}\\
 F_{\eta,Q}^k&\defn\set{A\in M_{m,n}(\bR):\frac{1}{k}\displaystyle\sum_{i=0}^{k-1}\del_{a^i x_A}
 (X\setminus Q)<\eta}.
\end{align*}
Take a sequence of increasing compact sets $\set{Q_j}$ exhausting $X$. Observe that $M_{m,n}(\bR)\setminus E_\eta=\displaystyle\bigcup_{j=1}^\infty F_{\eta,Q_j}$.

We denote by $\Bar{Y}$ the one-point compactification of $Y$ with $\sigma$-algebra $\cB_{\Bar{Y}}$ generated by $\cB_Y$ and $\set{\infty}$. The diagonal action $a_t$ is extended to the action on $\Bar{Y}$ by $a_t(\infty)=\infty$ for $t\in\bR$. For a finite partition $\cP=\set{P_1,\cdots,P_N,P_\infty}$ of $Y$ which has only one non-compact element $P_\infty$, denote by $\overline{\cP}$ the finite partition $\set{P_1, \cdots, P_N, \overline{P_\infty}\defn P_\infty\cup\set{\infty}}$ of $\Bar{Y}$. For any countable partition $\xi$, we denote by $\xi^{(q)}=\displaystyle\bigvee_{i=0}^{q-1}a^{-i}\xi$ the join of the preimages $a^{-i}\xi$. Note that $\overline{\cP^{(q)}}=\overline{\cP}^{(q)}$ for $q\in\N$.

For any countable partition $\cP$ of $Y$, $H_\mu(\cP|\cB_Y^U)$ will denote the relative entropy of $\cP$ with respect to the $\sigma$-algebra $\cB_Y^U$. Also denote by $h_\mu(a|\cB_Y^U)$ the relative entropy of the transformation $a$ for $\mu$.  Note also that $H_{\mu}(\cP^{(q)}|\cB_Y^U)=H_{\mu}({\overline{\cP}^{(q)}}|\cB_{\Bar{Y}}^U)$ for $\mu\in\crly{P}(Y)$. Here, $\cB_{\Bar{Y}}^U$ is the Borel $\sigma$-algebra generated by the collection of $U$-invariant Borel sets in $\Bar{Y}$. For the rest of the section, we construct the desired measure on $\Bar{Y}$ in Proposition \ref{prop2}. The construction will basically follow the construction in the \cite{LSS}, Section 2. However, the additional step using Theorem \ref{KKLM} is necessary to control the escape of mass since we will allow a small amount of escape of mass. 

\begin{prop}\label{prop2}
Assume that $\dim_H \mb{Bad}^{'}(\eps)>\dim_H \mb{Bad}^{0}(\eps)$. For $\gamma>0$, let $\eta=\frac{mn}{m+n}(mn-\dim_H \mb{Bad}^{'}(\eps)+\gamma)>0$. Then there exist an $a$-invariant probability measure $\mu\in \crly{P}(\Bar{Y})$ satisfying:
\begin{enumerate}
    \item\label{supp} $\Supp{\mu}\subseteq  \overline{\cL_\eps} = \cL_\eps \cup \{ \infty \}$,
    \item\label{cusp} $\mu(\Bar{Y}\setminus Y)\leq \eta$,
    \item\label{entropy} Let $K\subset Y$ be a compact set. If $\cP$ is any finite partition of $Y$ satisfying:
    \begin{itemize}
        \item $\cP$ contains an atom $P_\infty$ where $K \subseteq Y\smallsetminus P_\infty$,
        \item $\forall P\in \cP\smallsetminus\set{P_\infty}$, $\diam P<r_0$ for some $0<r_0<\frac{1}{2}$,
        \item $\forall P\in\cP, \mu(\partial P)=0,$
    \end{itemize}
    then, for all $q\ge 1$,
    $$ \frac{1}{q}H_{\mu}(\overline{\cP}^{(q)}|\cB_{\Bar{Y}}^U)\ge(1-\mu(Y\setminus K)^{\frac{1}{2}})(m+n)(\dim_H \mb{Bad}^{'}(\eps)-\gamma-mn\mu(Y\setminus K)^{\frac{1}{2}}).$$
\end{enumerate}
\end{prop}
In particular, by taking a sequence of compact sets $K_i$ exhausting $X$ and a sequence of positive numbers $\gamma_i$ which converges to $0$,
 $$h_{\mu} (a|\cB_{\Bar{Y}}^U) \geq (m+n) \dim_H \mb{Bad}^{'}(\eps).$$
 Note that the right handside is the maximal entropy if $\dim_H \mb{Bad}^{'}(\eps) = mn.$
\begin{proof}
Denote by $R$ the set $\mb{Bad}^{'}(\eps)\setminus \mb{Bad}_{(1)}^{'}(\eps)$, and let $R^{T}$ be the increasing set $\set{A\in R \subset M_{m,n}(\bR) |\forall t\ge T, a_t x_{A,b}\in \cL_\eps}$. By Proposition~\ref{prop1}, $R=\displaystyle\bigcup_{T=1}^\infty R^{T}$. Since $\dim_H \mb{Bad}^{'}(\eps)>\dim_H \mb{Bad}^{0}(\eps)\ge\dim_H \mb{Bad}_{(1)}^{'}(\eps)$,
$\dim_H R=\dim_H \mb{Bad}^{'}(\eps)$. Thus for any $\gamma>0$, there exists $T_\gamma>0$ satisfying
$$ \dim_H R^{T_\gamma}>\dim_H \mb{Bad}^{'}(\eps)-\gamma.$$
By Theorem~\ref{KKLM}, we obtain $\dim_H E_\eta \leq mn-\frac{\eta (m+n)}{mn}=\dim_H \mb{Bad}^{'}(\eps)-\gamma$ by definition of $\eta$.  
From the fact that $M_{m,n}(\bR)\setminus E_\eta=\displaystyle\bigcup_{j=1}^\infty F_{\eta,Q_j}$, we obtain $$\dim_H(\displaystyle\bigcup_{j=1}^\infty (R^{T_\gamma}\cap F_{\eta,Q_j}))=\dim_H(R^{T_\gamma}\setminus E_\eta)>\dim_H \mb{Bad}^{'}(\eps)-\gamma,$$ thus we can take a compact set $Q\subset X$ from the above sequence of increasing compact sets $\set{Q_j}\rightarrow X$ which satisfies  $$\dim_H(R^{T_\gamma}\cap F_{\eta,Q})>\dim_H \mb{Bad}^{'}(\eps)-\gamma.$$ Since $F_{\eta,Q}=\displaystyle\bigcap_{N=1}^\infty \displaystyle\bigcup_{k=N}^{\infty}F_{\eta,Q}^k=\displaystyle\limsup_{k\to\infty}F_{\eta,Q}^k$,
$$\dim_H (R^{T_\gamma}\cap F_{\eta, Q}^{k_i})> \dim_H \mb{Bad}^{'}(\eps)-\gamma$$
for an increasing sequence of positive integers $\set{k_i}\to \infty$ \cite{falconer}.

For a bounded subset $S\subseteq Y$, let $N_{\bf{d}}(S,\del)$ be the maximal cardinality of a $\del$-separated subset of $S$ for the metric $\bf{d}$. Then
$$\dim_H S\leq \displaystyle\liminf_{\del\to 0}\frac{\log N_{\bf{d}}(S,\del)}{\log \frac{1}{\del}}$$
(see section 2.2 of \cite{LSS}). Let $\phi: M_{m,n}\rightarrow X$ be the function defined by $\phi(A)=x_{A}$. For each $A\in R^{T_\gamma}$, there exists $b_A$ such that $a_t x_{A,b_A}\in\cL_\eps$ for all $t\geq T_\gamma$. For each $k_i\ge T_\gamma$ let $S_i$ be a maximal $e^{-k_i(m+n)}$-separated subset of $S_i^{\circ}$ with respect to the metric $\bf{d}$, where $S_i^{\circ}\defn\set{x_{A,b_A}:A\in R^{T_\gamma}\cap F_{\eta,Q}^{k_i}}\subset Y$. Note that $\pi(S_i^{\circ})=\phi(R^{T_\gamma}\cap F_{\eta,Q}^{k_i})$ and $\pi$ is Lipschitz.
Then 
\eqlabel{eqn1}{\begin{aligned}
  \displaystyle\liminf_{i\to\infty}\frac{\log |S_i|}{k_i}
  &\ge (m+n)\displaystyle\liminf_{\del\to 0}\frac{\log N_{\bf{d}}(S_i^{\circ},\del)}{\log \frac{1}{\del}}\\
  &\ge (m+n)\dim_H \phi(R^{T_\gamma}\cap F_{\eta, Q}^{k_i})\\
  &=(m+n)\dim_H (R^{T_\gamma}\cap F_{\eta, Q}^{k_i})\\
  &> (m+n)(\dim_H \mb{Bad}^{'}(\eps)-\gamma).\end{aligned}}
In \eqref{eqn1}, the first line is by definition, the second line is using the Lipschitz property of $\pi$, and the third line holds since $\phi$ is locally bi-Lipschitz from the local bi-Lipschitz property between $\bf{d}$ and $||\cdot||_{\bR^{mn}}$.

Let $\nu_i\defn \frac{1}{|S_i|}\displaystyle\sum_{y\in S_i}\delta_y$ be the normalized counting measure on $S_i$ and let
$$\mu_i\defn \frac{1}{k_i}\displaystyle\sum_{k=0}^{k_i-1}a^k_{*}\nu_i \wstar \mu\in\crly{P}(\Bar{Y})$$
By extracting a subsequence if necessary, there exists a probability measure $\mu$ which is a weak*-accumulation point of $\set{\mu_i}$. Now we prove that the measure $\mu$ is the desired measure. The measure $\mu$ is clearly an $a$-invariant measure.\\ 
(\ref{supp}) For any $y\in S_i\subseteq \pi^{-1}\circ\phi(R^{T_\gamma})$, $a^T y\in \cL_\eps$ holds for $T>T_\gamma$. Thus
$$\mu_i(Y\setminus\cL_\eps)=\frac{1}{k_i}\displaystyle\sum_{k=0}^{k_i-1}a^k_{*}\nu_i(Y\setminus\cL_\eps)=\frac{1}{k_i}\displaystyle\sum_{k=0}^{T_\gamma}a^k_{*}\nu_i(Y\setminus\cL_\eps)\leq\frac{T_\gamma}{k_i}.$$ We obtain item (\ref{supp}) by taking limit for $k_i\to \infty$.\\
(\ref{cusp}) Let $K\subset Y$ be the compact set which is defined as $K=\pi^{-1}(Q)$. If $y\in S_i\subset \pi^{-1}\circ\phi(F_{\eta,Q}^{k_i})$, for all $i \in \mathbb{N}$,  $\frac{1}{k_i}\displaystyle\sum_{k=0}^{k_i-1}\delta_{a^k y}(Y\setminus K)<\eta$. Therefore for all $i$,
\begin{align*}
    \mu_i(Y\setminus K)
    =\frac{1}{|S_i|}\displaystyle\sum_{y\in S_i} \frac{1}{k_i}\displaystyle\sum_{k=0}^{k_i-1} \delta_{a^k y}(Y\setminus K)<\eta,
\end{align*}
thus
$$\mu(\bar{Y}\setminus Y)\leq \mu(\bar{Y}\setminus K)=\displaystyle\lim_{i\to\infty}\mu_i(Y\setminus K)\leq \eta.$$

For the rest of the proof, let us check the condition (\ref{entropy}). Let $K$ be a compact set such that $\mu(Y\backslash K) <1$ (otherwise the condition (\ref{entropy}) is trivial). \\
(\ref{entropy})
Let $\rho>0$ be small enough so that $\beta : = \mu(Y\setminus K)+\rho < 1$. Then 
$$\beta = \mu(Y\setminus K)+\rho>\mu_i(Y\setminus K)=\frac{1}{k_i|S_i|}\displaystyle\sum_{y\in S_i, 0\le k<k_i}\de_{a^k y}(Y\setminus K)$$ holds for large enough $i$. In other words, there exist at most $\beta k_i|S_i|$ number of $a^k y$'s in $Y\setminus K$.

Let $S'_i\subset S_i$ be the set of $y$'s satisfying
$a^k y\in K$ for some $k\in [(1-\beta^{\frac{1}{2}})k_i,k_i)$. Then we have $|S'_i|\ge (1-\beta^{\frac{1}{2}})|S_i|$ since
$$ (1-\beta) k_i |S_i| \leq |\{ a^ky: 0< k<k_i, y \in S_i \} |\leq \sqrt \beta k_i |S_i'|.$$ 
Let $\nu'_i\defn \frac{1}{|S'_i|}\displaystyle\sum_{y\in S'_i}\de_y$ be the normalized counting measure on $S'_i$, then $$\nu_i(A)\ge \frac{|S'_i|}{|S_i|}\nu'_i(A)$$ for all measurable set $A\subseteq Y$. Thus, for any arbitrary countable partition $\cP$, 
\eqlabel{eqn2}{
\begin{aligned}
H_{\nu_i}(\cP)&=-\displaystyle\sum_{\nu_i(A)\leq\frac{1}{e}}\log(\nu_i(A))\nu_i(A)-\displaystyle\sum_{\nu_i(A)>\frac{1}{e}}\log(\nu_i(A))\nu_i(A)\\&\ge -\displaystyle\sum_{\nu_i(A)\leq\frac{1}{e}}\log(\frac{|S'_i|}{|S_i|}\nu'_i(A))\frac{|S'_i|}{|S_i|}\nu'_i(A)
\\&=-\frac{|S'_i|}{|S_i|}\displaystyle\sum_{\nu_i(A)\leq\frac{1}{e}}\log(\nu'_i(A))\nu'_i(A)-\frac{|S'_i|}{|S_i|}\log{\frac{|S'_i|}{|S_i|}}\displaystyle\sum_{\nu_i(A)\leq\frac{1}{e}}\nu'_i(A)
\\&\geq \frac{|S'_i|}{|S_i|}\Bigl\{H_{\nu'_i}(\cP)+\displaystyle\sum_{\nu_i(A)>\frac{1}{e}}\log(\nu'_i(A))\nu'_i(A)\Bigr\}\\&\ge (1-\beta^{\frac{1}{2}})(H_{\nu'_i}(\cP)-\frac{3}{e}).
\end{aligned}
}
In the last inequality, we use the fact that $\nu'_i$ is a probability measure, thus there can be at most three elements $A$ of the partition for which $\nu'_i (A) > \frac{1}{e}$.
 If $P$ is any non-empty atom of $\cP^{(k_i)}$, fixing any $y_0\in P$, any $y\in S'_i\cap P=S'_i\cap [y_0]_{\cP^{k_i}}$ satisfies
\begin{align*}
  r_0>\mathbf{d}(a^k y_0, a^k y)\ge C'e^{(m+n)k}\mathbf{d}(y_0,y)\ge C'e^{(m+n)(1-\beta^{\frac{1}{2}})k_i}\mathbf{d}(y_0,y)  
\end{align*}
for some constant $C'>0$.
Here, we used the right invariant property of $\bf{d}$ and bi-Lipshitz property between $\bf{d}$ and $||\cdot||$. Thus $S'_i\cap P$ can be covered by one ball of $C'^{-1}r_0e^{-(m+n)(1-\beta^{\frac{1}{2}})k_i}$-radius for metric $\bf{d}$ as well as by $C_1 e^{-(m+n)mn\beta^{\frac{1}{2}} k_i}$ many balls of $r_0e^{-(m+n)k_i}$-radius for the metric $\bf{d}$ and some constant $C_1>0$. Since $S'_i$ is $e^{(m+n)k_i}$-separated with respect to $\bf{d}$, we get 
$$Card(S'_i \cap [y_0]_{\cP^{(k_i)}})\leq C_1 e^{(m+n)mn\beta^{\frac{1}{2}} k_i},$$
and therefore 
$$H_{\nu'_i}(\cP^{(k_i)})\ge \log |S'_i|-(m+n)mn\beta^{\frac{1}{2}} k_i-\log C_1.$$
Now we can estimate the lower bound of the entropy. For $q\ge 1$, write the Euclidean division of large enough $k_i-1$ by $q$ as
$$k_i-1=qk'+s \ \textrm{with} \ s\in\set{0,\cdots,q-1}.$$
By subadditivity of the entropy with respect to the partition, for each $p\in\set{0,\cdots,q-1}$,
$$H_{\nu_i}(\cP^{(k_i)}|\cB_Y^U)\leq H_{a^{p}\nu_i}(\cP^{(q)}|\cB_Y^U)+\cdots+H_{a^{p+qk'}\nu_i}(\cP^{(q)}|\cB_Y^U)+2q\log |\cP|.$$
Summing those inequalities for $p=0,\cdots,q-1$, and using the concave property of entropy with respect to the measure, we obtain
\begin{align*}
  qH_{\nu_i}(\cP^{(k_i)}|\cB_Y^U)
  &\leq\displaystyle\sum_{k=0}^{k_i-1}H_{a^k \nu_i}(\cP^{(q)}|\cB_Y^U)+2q^2\log |\cP|\\
  &\leq k_iH_{\mu_i}(\cP^{(q)}|\cB_Y^U)+2q^2\log |\cP|
\end{align*}
and therefore
\begin{align*}
    \frac{1}{q}H_{\mu_i}(\cP^{(q)}|\cB_Y^U)
    &\ge \frac{1}{k_i}H_{\nu_i}(\cP^{(k_i)}|\cB_Y^U)-\frac{2q\log |\cP|}{k_i}\\
    &\ge \frac{1}{k_i}\Bigl\{(1-\beta^{\frac{1}{2}})(H_{\nu'_i}(\cP^{(k_i)})-\frac{3}{e})-2q\log|\cP|\Bigr\}\\
    &\ge \frac{1}{k_i}\Bigl\{(1-\beta^{\frac{1}{2}})(\log |S'_i|-mn(m+n)\beta^{\frac{1}{2}} k_i -\log C_1-\frac{3}{e}) - 2q\log|\cP|\Bigr\}
\end{align*}
Here, for the second inequality, we used inequality \eqref{eqn2} and the fact that $$H_{\nu_i}(\cP^{(q)}|\cB_Y^U)=H_{\nu_i}(\cP^{(q)})$$ which holds as $\nu_i$ is supported on an atom of $\cB_Y^U$. Now we can take $i\to\infty$ because the atoms $P$ of $\cP$ and hence of $\cP ^{(q)}$, satisfy $\mu(\partial P)=0$. Thus we obtain the inequality
\begin{align*}
    \frac{1}{q}H_{\mu}(\overline{\cP}^{(q)}|\cB_{\Bar{Y}}^U)
    &\ge (1-\beta^{\frac{1}{2}})(m+n)\{(\dim_H \mb{Bad}^{'}(\eps)-\gamma)-mn\beta^{\frac{1}{2}}\},
\end{align*}
from the inequality \eqref{eqn1} and finally get the inequality
$$\frac{1}{q}H_{\mu}(\overline{\cP}^{(q)}|\cB_{\Bar{Y}}^U)
    \ge (1-\mu(Y\setminus K)^{\frac{1}{2}})(m+n)\{(\dim_H \mb{Bad}^{'}(\eps)-\gamma)-mn\mu(Y\setminus K)^{\frac{1}{2}}\}$$
we desired by taking $\rho \to 0$.
\end{proof}

\section{Proof of main results}
\subsection{Maximal entropy implies invariance}
To deduce the invariance property of the measure we constructed in Section 2, we need the following proposition about maximal entropy.
\begin{prop}[Maximal entropy implies $U$-invariance]
\label{prop3}
Let $\mu$ be an $a_t$-invariant probability measure on $Y$.
Then 
\[ h_\mu(a|\cB_Y^U) \leq \log|det(Ad_a|_\bu)| \]
with equality if and only if $\mu$ is $U$-invariant.
\end{prop}
Note that for our situation $a_t=\Delta (e^{nt}\mathbf{1}_m,e^{-mt}\mathbf{1}_n)$ in $SL_d(\bR)$, the restriction of the adjoint map $Ad_a|_\bu$ can be considered as the map $A\rightarrow e^{-(m+n)}A$ for $A\in M_{m,n}(\bR)$ by identifying $\bu \simeq M_{m,n}(\bR)$, so the maximal entropy is $\log|det(Ad_a|_\bu)|=(m+n)mn$.

\begin{defi}[7.25. of \cite{EL}]
Let $G^-\defn \set{g\in G|a_t g a_t^{-1}\to e \ \textrm{as} \ t\to\infty}$ be the stable horospherical subgroup associated to $a$. Let $\mu$ be an $a$-invariant measure on $Y$ and $U<G^-$ be a closed $a$-normalized subgroup.
\begin{enumerate}
\item We say that a countably generated $\s$-algebra $\cA$ is \emph{subordinate to $U^+$} (mod $\mu$) if for $\mu$-a.e. $y$, there exists $\de > 0$ such that
$$ B^{U^+}_\de\cdot y \subset [y]_{\cA} \subset B^{U^+}_{\de^{-1}}\cdot y.$$ 
\item
We say that $\cA$ is \emph{$a$-descending} if $a^{-1}\cA \subset \cA$.
\end{enumerate}
\end{defi}

The proof of Proposition \ref{prop3} is based on the following theorem applied to $a^{-1}$ so that $U<G^-$. To obtain following theorem, it is enough to consider the case that $\mu$ is ergodic by using ergodic decomposition. Combining Proposition 7.34 and Theorem 7.9 of \cite{EL}, we obtain the following statement under the ergodicity assumption.
\begin{thm}[Einsiedler-Lindenstrauss]\label{thmEL}
Let $\mu$ be an $a_t$-invariant ergodic probability measure on $Y$. If $\cA$ is a countably generated sub-$\sigma$-algebra of the Borel $\sigma$-algebra which is $a$-descending and $U$-subordinate, then $$H_\mu(\cA|a^{-1}\cA)\leq -\log|det(Ad_a|_\bu)|.$$
\end{thm}

For detailed proof of Proposition \ref{prop3} using Theorem \ref{thmEL}, we refer the reader to \cite{LSS}, Chapter 3. The only difference here is that we calculate relative entropy with respect to the $\sig$-algebra $\cB_Y^U$ instead of the $\sig$-algebra $\pi^{-1}(\cB_X)$ in \cite{LSS}.



\subsection{Proof of Theorem \ref{thm0}}
In this subsection, to prove Theorem \ref{thm0}, we investigate the closed set $\cL_\eps$. Measures supported on the set $\cL_\eps$ only admit few invariance properties, as stated in the following proposition.
\begin{prop}\label{prop4}
Let $\mu\in\cP(Y)$ be a measure which is $a_t$-invariant and $U$(or $W$)-invariant. Then $\mu$ cannot be supported on $\cL_\eps$ for any $\eps>0$.  
\end{prop}
\begin{proof}
Assume that $\Supp\mu \subseteq \cL_\eps$ for some $\eps>0$, then $\mu(Y\setminus \cL_\eps)=0$ holds. First we claim that for any $y\in Y$, ther exist $t\in \bR$ and $u\in U$ satisfying $a_t uy\in Y\setminus \cL_\eps$. Let $y=
\left(\begin{matrix}
X_1 & X_2 & b_1\\
X_3 & X_4 & b_2\\
0 & 0 & 1\\
\end{matrix}\right)G(\bZ)
$, then for $u=u(A)$ and
$a_t$, we have
\begin{align*}
    a_t uy&=
    \left(\begin{matrix}
e^{nt}I_m & 0 & 0 \\
0 & e^{-mt}I_n &  0\\
0 & 0 & 1\\
\end{matrix}\right)
\left(\begin{matrix}
I_m & A & 0\\
0 & I_n & 0\\
0 & 0 & 1\\
\end{matrix}\right)
\left(\begin{matrix}
X_1 & X_2 & b_1\\
X_3 & X_4 & b_2\\
0 & 0 & 1\\
\end{matrix}\right)G(\bZ)\\
&=
\left(\begin{matrix}
e^{nt}X_1 & e^{nt}(X_2+AX_4) & e^{nt}(b_1+Ab_2)\\
e^{-mt}X_3 & e^{-mt}X_4 & e^{-mt}b_2\\
0 & 0 & 1\\
\end{matrix}\right)G(\bZ).
\end{align*}
We consider two cases to prove the claim:\\
$\mathbf{Case\ 1)}$ $b_2\neq 0$.\\
In this case, we can find a matrix $A\in M_{m,n}(\bR)$ satisfying $b_1+Ab_2=0$. 
Then the translation vector of the grid from the origin, which is the first $m+n$ entries of the last column, is $(\mathbf{0}_m,e^{-mt}b_2)$. Thus the vector $(\mathbf{0}_m,e^{-mt}b_2)$ is contained in the grid $\Gamma_{a_t uy}$: By taking $t\to\infty$, we can choose $t>>0$ with $||e^{-mt}b_2||<\eps^{\frac{m}{m+n}}$, which gives $\Gamma_{a_t uy}\cap B_{\eps^{\frac{m}{m+n}}}^{\bR^d}(0)\neq \phi$.\\
$\mathbf{Case\ 2)}$ $b_2= 0$.\\
In this case, the last column is $(e^{nt}b_1,\mathbf{0}_n)$. We can make this translation vector small enough by taking $t\to -\infty$. Thus we get $\Gamma_{a_t uy}\cap B_{\eps^{\frac{m}{m+n}}}^{\bR^d}(0)\neq \phi$ as in the Case 1.\\
Combining these two cases, we arrived at the claim.\\
Similarly, for any $y\in Y$, there exist $t\in\bR$ and $w\in W$ satisfying $a_t wy\in Y\setminus\cL_\eps$.
Indeed, for $w\in W$, let $w=
\left(\begin{matrix}
I_m & 0 & c\\
0 & I_n & 0\\
0 & 0 & 1\\
\end{matrix}\right)$, then
$$a_t wy=\left(\begin{matrix}
e^{nt}X_1 & e^{nt}X_2 & e^{nt}(b_1+c)\\
e^{-mt}X_3 & e^{-mt}X_4 & e^{-mt}b_2\\
0 & 0 & 1\\
\end{matrix}\right)G(\bZ)$$ holds and we can apply previous argument by taking $c=-b_1$ and $t\to\infty$.

Since $\cL_\eps$ is closed, for every $y\in\bR$, there exist $r_y>0$, $t\in\bR$, and $u\in U$ such that the $d$-ball $B_{r_y}(a_t uy)\subseteq Y\setminus \cL_\eps$. Choose $r'_y>0$ such that $B_{r'_y}(y)\subseteq (a_t u)^{-1}B_{r_y}(a_t uy)$. Then for  $a_t$-invariant and $U$-invariant measure $\mu$,
\begin{align*}
    \mu(B_{r'_y}(y))\leq \mu((a_t u)^{-1}B_{r_y}(a_t uy))
    =\mu(B_{r_y}(a_t uy))
    \leq \mu(Y\setminus \cL_\eps)=0.
\end{align*}
Covering $Y$ by balls $B_{r'_y}(y)$, we obtain 
$$\mu(Y)=\mu(\displaystyle\bigcup_{y\in Y}B_{r'_y}(y))\leq 0,$$
which gives a contradiction. We also prove the $W$-invariant case by replacing $u\in U$ with $w\in W$.

\end{proof}

\begin{proof}[Proof of Theorem \ref{thm0}]
Suppose that for some fixed $\eps>0$, $\dim_H \mb{Bad}^{b_j}(\eps)=mn$. Take a sequence of real numbers $\set{\gamma_j}$ which converge to zero. Now apply Proposition \ref{prop2} for $\gamma_j>0$. Note that the corresponding $\eta_j=\frac{mn}{m+n}(mn-\dim_H \mb{Bad}^{b_j}(\eps)+\gamma_j)=\frac{mn \gamma_j}{m+n}$ goes to zero as $j\to\infty$. Then we have a sequence of probability measures $\set{\mu_j}$ in $\crly{P}(\Bar{Y})$ satisfying the four conditions in Proposition \ref{prop2}. Now, there exists a probability measure $\mu\in\crly{P}(\bar{Y})$ which is a weak*-accumulation point of $\set{\mu_j}$. Let
$$\mu_j\wstar\mu\in\crly{P}(\bar{Y})$$
by extracting a subsequence if necessary.
Clearly $\mu$ is $a$-invariant and $\Supp\mu\subseteq \cL_\eps$. Since $\mu(\bar{Y}\setminus Y)=\displaystyle\lim_{j\to\infty}\mu_j(\bar{Y}\setminus Y)=\displaystyle\lim_{j\to\infty}\eta_j=0$, we may consider $\mu$ as a probability measure on $Y$. 

Let us show that $\mu$ has maximal entropy. For any compact set $K\subset Y$, we can build a finite partition $\cP$ satisfying:
\begin{enumerate}
        \item\label{item1} $\cP$ contains an atom $P_\infty$ where $K \subseteq Y\smallsetminus P_\infty$,
        \item\label{item2} $\forall P\in \cP\smallsetminus\set{P_\infty}$, $\diam P<r_0$ for some $0<r_0<\frac{1}{2}$,
        \item\label{item3} $\forall P\in\cP \ \textrm{and}\ j\in\bN, \mu_j(\partial P)=0$.
    \end{enumerate}
It is possible to build this partition to satisfy the part (\ref{item3}) because the collection of measure $\set{\mu_j}$ is countable. Thus we have
$$ \frac{1}{q}H_{\mu_j}(\Bar{\cP}^{(q)}|\cB_{\Bar{Y}}^U)\ge(1-\mu_j(Y\setminus K)^{\frac{1}{2}})(m+n)(\dim_H \mb{Bad}^{'}(\eps)-\gamma_j-mn\mu_j(Y\setminus K)^{\frac{1}{2}})$$
for all $q,j\in\bN$. Since $\cP^{(q)}$ is finite, we can take $j\to\infty$ to obtain
$$\frac{1}{q}H_{\mu}(\cP^{(q)}|\cB_Y^U)\ge(1-\mu(Y\setminus K)^{\frac{1}{2}})^2 (m+n)mn$$
for all $q\in\bN$. Thus $h_\mu(a|\cB_Y^U)\ge(1-\mu(Y\setminus K)^{\frac{1}{2}})^2 (m+n)mn$ holds for any compact set $K\subset Y$, and eventually we have $h_\mu(a|\cB_Y^U)=(m+n)mn$ which is the maximal entropy with respect to $\cB_Y^U$. It follows that $\mu$ is $U$-invariant by Proposition \ref{prop3}. We arrived at the desired contradiction by using proposition \ref{prop4} since we showed that $\mu$ is a probability measure on $Y$ which is $a_t$-invariant, $U$-invariant, and supported on $\cL_\eps$.
\end{proof}

\subsection{Proof of Theorem \ref{thm2} and \ref{thm3}}\label{sec:5}
Before we start the proof of Theorem \ref{thm2} and \ref{thm3}, we construct a measure on $Y$ with large relative entropy as we did in Proposition \ref{prop2}. However, in this case, we will calculate entropy relative to the factor $X$, instead of the $\sigma$-algebra $\cB_Y^U$ we used in previous chapters. For any coutable partition $\cP$ of $Y$, $H_\mu(\cP|X)$ will denote the relative entropy of $\cP$ with respect to the $\sigma$-algebra $\pi^{-1}(\cB_X)$ where $\cB_X$ is the Borel $\sigma$-algebra on $X$. Similarly, for $\mu\in\crly{P}(\Bar{Y})$, $H_\mu(\cP|\Bar{X})$ will denote the relative entropy with respect to the $\pi^{-1}(\cB_{\Bar{X}})$ where $\cB_{\Bar{X}}$ is the Borel $\sigma$-algebra on $\Bar{X}$. Note that $H_\mu(\cP^{(q)}|X)=H_\mu(\overline{\cP}^{(q)}|\Bar{X})$ for $q\in\bN$ and $\mu\in\crly{P}(Y)$. Also denote by $h_\mu(a|X)$ and $h_\mu(a|\Bar{X})$ the relative entropy of the transformation $a$ with respect to $X$ and $\Bar{X}$, respectively. The following proposition is analogous to Proposition \ref{prop2} as well as its proof except a slight difference on the construction of measure.

\begin{prop}\label{prop5} For $A\in M_{m,n}(\R)$ fixed,
let $$\eta_0=\sup\set{\eta:x_A \ \textrm{has} \ \eta\textrm{-escape of mass on average}}.$$ Then for any $\eps>0$ and $\gamma>0$, there exist an $a_t$-invariant probability measure $\mu^\gamma\in \crly{P}(\bar{Y})$ satisfying:
\begin{enumerate}
    \item\label{supp'} $\Supp{\mu^\gamma}\subseteq \cL_\eps$,
    \item\label{cusp'} $\mu^\gamma(\Bar{Y}\setminus Y)=\eta_0$,
    \item\label{entropy'}If $\cP$ is any finite partition of $Y$ satisfying:
    \begin{itemize}
        \item $\cP$ contains an atom $P_\infty$ of the form $\pi^{-1}(P_\infty^0)$, where $X\smallsetminus P_\infty^0$ has compact closure,
        \item $\forall P\in \cP\smallsetminus\set{P_\infty}$, $\diam P<r_0$ for some $0<r_0<\frac{1}{2}$,
        \item $\forall P\in\cP, \mu^\gamma(\partial P)=0$,
    \end{itemize}
    then, for all $q\ge 1$,
    $$\frac{1}{q}H_{\mu^\gamma}(\overline{\cP}^{(q)}|\Bar{X})\ge n(\dim_H \mb{Bad}_A(\eps)-\gamma)-mn\mu^\gamma(\overline{P_\infty}).$$
 
\end{enumerate}
\end{prop}
\begin{proof}
$R^{A,T}:=\set{b\in\bR^m|\forall t\ge T,a_t x_{A,b}\in\cL_\eps}\cap\mb{Bad}_A(\eps)$. By Proposition \ref{prop1}, \\$\displaystyle\bigcup_{T=1}^\infty R^{A,T}$ has Hausdorff dimension equal to $\dim_H \mb{Bad}_A(\eps)$, thus for any $\gamma >0$, there exists $T_\gamma>0$ satisfying $\dim_H R^{A,T_\gamma}\geq \dim_H \mb{Bad}_A(\eps)-\gamma$.
Since $A$ has $\eta$-escape of mass on average for $\eta<\eta_0$, we may fix an increasing sequence of integers $\set{k_i}$ such that
$$\frac{1}{k_i}\displaystyle\sum_{k=0}^{k_i-1}\del_{a^k x_A}\wstar\mu_A\in\crly{P}(\Bar{X})$$ with $\mu_A(X)= 1-\eta_0$.
 Let $\phi_A:\bR^m\rightarrow Y$ be the (bi-Lipshitz) function defined by $\phi_A(b)=x_{A,b}$. For each $k_i\ge T_\gamma$, let $S_i$ be a maximal $e^{-k_i n}$-separated subset of $\phi_A(R^{A,T_\gamma})$ with respect to the metric $\bf{d}$. Then
$$\displaystyle\liminf_{i\to\infty}\frac{\log|S_i|}{k_i}\ge n\dim_H(R^{A,T_\gamma})\ge n(\dim_H \mb{Bad}_A(\eps)-\gamma)$$
holds from the bi-Lipschitz property between $\bf{d}$ and $||\cdot||_{\bR^m}$. \\Let $\nu_i\defn\frac{1}{|S_i|}\displaystyle\sum_{y\in S_i}\del_y$ be the normalized counting measure on $S_i$ and let $$\mu_i\defn\frac{1}{k_i}\displaystyle\sum_{k=0}^{k_i-1}a_*^k\nu_i\wstar\mu^\gamma\in\crly{P}(\bar{Y}).$$
we prove that $\mu^\gamma$ is the desired measure.\\
(\ref{supp'}) For any $y\in S_i\subseteq \phi_A(R^{A,T_\gamma})$, $a^T y\in \cL_\eps$ holds for $T>T_\gamma$. Thus
$$\mu_i(Y\setminus\cL_\eps)=\frac{1}{k_i}\displaystyle\sum_{k=0}^{k_i-1}a^k_{*}\nu_i(Y\setminus\cL_\eps)=\frac{1}{k_i}\displaystyle\sum_{k=0}^{T_\gamma}a^k_{*}\nu_i(Y\setminus\cL_\eps)\leq \frac{T_\gamma}{k_i}$$ and we obtain item (\ref{supp'}) by taking limit for $k_i\to \infty$.\\
(\ref{cusp'}) Since $\pi_*\nu_i=\del_{x_A}$ for all $i\ge 1$, $\pi_*\mu^\gamma=\mu_A$ holds. Thus $$\mu^\gamma(\Bar{Y}\setminus Y)=\mu_A(\Bar{X}\setminus X)=\eta_0.$$\\
(\ref{entropy'})
By part \ref{supp'}, $\mu^\gamma P_\infty)<1$ for small enough $P_\infty$.
Let $\rho>0$ be small enough so that $\beta\defn\mu^\gamma(P_\infty)+\rho<1$.  Then for large enough $i$, 
\begin{align*}
\beta= \mu^\gamma(P_\infty)+\rho>\mu_i(P_\infty)&=\frac{1}{k_i|S_i|}\displaystyle\sum_{y\in S_i, 0\le k<k_i}\de_{a^k y}(P_\infty)\\
&=\frac{1}{k_i}\displaystyle\sum_{0\le k<k_i}\de_{a^k x_A}(P^0_\infty)
\end{align*}
In other words, there exist at most $\beta k_i$ number of $a^k x_A$'s in $P^0_\infty$, thus there exists some $k\in [(1-\beta)k_i,k_i)$ such that 
$a^k y\in Y\setminus P_\infty$ for all $y\in S_i$.
 If $P$ is any non-empty atom of $\cP^{(k_i)}$, fixing any $y_0\in P$, any \\$y\in S_i\cap P=S_i\cap [y_0]_{\cP^{k_i}}$ satisfies
\begin{align*}
  r_0>\mathbf{d}(a^k y_0, a^k y)\ge C''e^{nk}\mathbf{d}(y_0,y)\ge C''e^{n(1-\beta)k_i}\mathbf{d}(y_0,y)  
\end{align*}
for some constant $C''>0$.
Here, we used the right invariant property of $\bf{d}$ and bi-Lipshitz property between $\bf{d}$ and $||\cdot||$. Thus $S_i\cap P$ can be covered by one ball of $C''^{-1}r_0e^{-n(1-\beta)k_i}$-radius for metric $\bf{d}$ as well as by $C_1 e^{mn\beta k_i}$ many balls of $r_0e^{-k_in}$-radius for the metric $\bf{d}$ and some constant $C_1>0$. Since $S_i$ is $e^{-k_in}$-separated with respect to $\bf{d}$, we get 
$$Card(S_i \cap [y_0]_{\cP^{(k_i)}})\leq C_1 e^{mn\beta k_i},$$
and therefore 
$$H_{\nu_i}(\cP^{(k_i)})\ge \log |S_i|-mn\beta k_i-\log C_1.$$
Now we can estimate the lower bound of the entropy. For $q\ge 1$, write the Euclidean division of large enough $k_i-1$ by $q$ as
$$k_i-1=qk'+s \ \textrm{with} \ s\in\set{0,\cdots,q-1}.$$
By subadditivity of the entropy with respect to the partition, for each $p\in\set{0,\cdots,q-1}$,
$$H_{\nu_i}(\cP^{(k_i)}|X)\leq H_{a^{p}\nu_i}(\cP^{(q)}|X)+\cdots+H_{a^{p+qk'}\nu_i}(\cP^{(q)}|X)+2q\log |\cP|.$$
Summing those inequalities for $p=0,\cdots,q-1$, and using the concave property of entropy with respect to the measure, we obtain
\begin{align*}
  qH_{\nu_i}(\cP^{(k_i)}|X)
  &\leq\displaystyle\sum_{k=0}^{k_i-1}H_{a^k \nu_i}(\cP^{(q)}|X)+2q^2\log |\cP|\\
  &\leq k_iH_{\mu_i}(\cP^{(q)}|X)+2q^2\log |\cP|
\end{align*}
and therefore
\begin{align*}
    \frac{1}{q}H_{\mu_i}(\cP^{(q)}|X)
    &\ge \frac{1}{k_i}H_{\nu_i}(\cP^{(k_i)}|X)-\frac{2q\log |\cP|}{k_i}\\
    &\ge \frac{1}{k_i}\Bigl\{(\log |S_i|-mn\beta k_i -\log C_1) - 2q\log|\cP|\Bigr\}
\end{align*}
Here, for the second inequality, we used $H_{\nu_i}(\cP^{(q)}|X)=H_{\nu_i}(\cP^{(q)})$ from the fact that $\nu_i$ is supported on an atom of $\pi^{-1}(\cB_X)$. Now we can take $i\to\infty$ because the atoms $P$ of $\overline{\cP}$ and hence of $\overline{\cP} ^{(q)}$, satisfy $\mu^\gamma(\partial P)=0$. Thus we obtain the inequality
\begin{align*}
    \frac{1}{q}H_{\mu^\gamma}(\overline{\cP}^{(q)}|\Bar{X})
    &\ge n(\dim_H \mb{Bad}_A(\eps)-\gamma)-mn\beta,
\end{align*}
and finally get the inequality
$$\frac{1}{q}H_{\mu^\gamma}(\overline{\cP}^{(q)}|\Bar{X})
    \ge n(\dim_H \mb{Bad}_A(\eps)-\gamma)-mn\mu^{\gamma}(\overline{P_\infty})$$
we desired by taking $\rho \to 0$.
\end{proof}

Now we prove Theorem \ref{thm2} and Theorem \ref{thm3}.
\begin{proof}[Proof of Theorem \ref{thm2}]
We use proof by contradiction again. Suppose that for any $\del>0$, the set of $A$ satisfying $\dim_H\mb{Bad}_A(\eps)\ge m-\del$ has full Hausdorff dimension. The set of $A$ satisfying $\dim_H \mb{Bad}_A(\eps)\ge m-\del$ has full Hausdorff dimension and for any $\eta>0$, the Hausdorff dimension of the set of $A$ with $\eta$-escape of mass on average is strictly less than full dimension by Theorem \ref{KKLM}. Thus there exists $A\in M_{m,n}(\bR)$ which does not have $\eta$-escape of mass on average and satisfies $\dim_H \mb{Bad}_A(\eps)\ge m-\del$ for any $\eta>0$.

 For $k\in\bN$, fix $\del=\eta=\frac{1}{k}$ and take a sequence of real numbers $\set{\gamma_j}$ which converges to zero. Then there exists a sequence of $a_t$-invariant probability measures $\set{\mu^{\gamma_j}\in\crly{P}(\bar{Y})}$ described in Proposition \ref{prop5}. By taking weak*-limit of this sequence, we can construct the measure $\mu_k\in\crly{P}(\Bar{Y})$ such that $\Supp\mu_k\subset\cL_\eps$, $\mu_k(\Bar{Y}\setminus Y)=\eta_0\le\eta$,  $$\frac{1}{q}H_{\mu_k}(\overline{\cP}^{(q)}|\Bar{X})\ge mn-\frac{n}{k}-mn\mu^\gamma(\overline{P_\infty})$$ for all $q\ge1$, and $\cP$ satisfying the condition (\ref{entropy'}) in Proposition \ref{prop5}. By taking weak*-limit of sequence $\set{\mu_k}$ again, eventually we obtain the measure $\mu\in\crly{P}(Y)$ supported on $\cL_\eps$ such that $\frac{1}{q}H_\mu(\cP^{(q)}|X)\ge (1-\mu(P_\infty)) mn$. By taking an appropriate partition $\cP$ to make $\mu(P_\infty)$ arbitrarily small, $h_\mu(a|X)=mn$, which is the maximal relative entropy, thus $\mu$ is a $W$-invariant measure by Proposition 3.1 of \cite{LSS}. We obtain the desired contradiction by Proposition \ref{prop4} since $\mu$ is a probability measure on $Y$ which is $a_t$-invariant, $W$-invariant, and supported on $\cL_\eps$. This proves Theorem \ref{thm2}.
\end{proof}

\begin{proof}[Proof of Theorem \ref{thm3}]
Suppose that $A\in M_{m,n}(\R)$ is not singular on average and satisfies $\dim_H\mb{Bad}_A(\eps)=m$. Let $$\eta_0=\sup\set{\eta:x_A \ \textrm{has} \ \eta\textrm{-escape of mass}}<1.$$ Take a sequence of real numbers $\set{\gamma_j}$ which converges to zero. Then there exists a sequence of $a_t$-invariant probability measures $\set{\mu^{\gamma_j}\in\crly{P}(\bar{Y})}$ described in Proposition \ref{prop5}. By taking weak*-limit of this sequence, we can construct the measure $\mu\in\crly{P}(\Bar{Y})$ such that $\Supp\mu\subset\cL_\eps\cup(\Bar{Y}\setminus Y)$,  $$\frac{1}{q}H_{\mu}(\overline{\cP}^{(q)}|\Bar{X})\ge mn(1-\mu(\overline{P_\infty}))$$ for all $q\ge1$, and $\cP$ satisfying the condition (\ref{entropy}) in Proposition \ref{prop5}. By taking the set $\overline{P^0_\infty } \subset \overline X$ smaller and smaller, we obtain $h_\mu(a|\Bar{X})\ge mn(1-\eta_0)$. 
On the other hand, the measure $\mu\in\crly{P}(\Bar{Y})$ can be represented by the linear combination $\mu=(1-\eta_0)\mu_0+\eta_0\del_\infty$, where $\del_\infty$ is the dirac delta measure on $\Bar{Y}\setminus Y$ and $\mu_0\in\crly{P}(Y)$. Since $s=\mu(\Bar{Y}\setminus Y)\le \eta_0$ and the entropy is a linear function with respect to the measure, 
$h_{\mu_0}(a|X)= \frac{1}{1-\eta_0}h_{\mu}(a|X)=mn$, which is the maximal relative entropy, thus $\mu_0$ is a $W$-invariant measure by Proposition 3.1 of \cite{LSS}. Proposition \ref{prop4} implies that $\mu_0$ is a probability measure on $Y$ which is $a_t$-invariant, $W$-invariant, thus containing full fibers, and supported on $\cL_\eps$. However a translation of a lattice by small $w$ cannot avoid an $\epsilon$-ball, thus we obtain the desired contradiction.
\end{proof}

\section*{Acknowledgement}   
SL is an associate member of KIAS. This project is supported by Samsung Science and Technology Foundation under Project No. SSTF-BA1601-03.

\def\cprime{$'$} \def\cprime{$'$} \def\cprime{$'$}
\providecommand{\bysame}{\leavevmode\hbox to3em{\hrulefill}\thinspace}
\providecommand{\MR}{\relax\ifhmode\unskip\space\fi MR }
\providecommand{\MRhref}[2]{%
  \href{http://www.ams.org/mathscinet-getitem?mr=#1}{#2}
}
\providecommand{\href}[2]{#2}

\end{document}